\documentclass[11pt, oneside]{article}   	
\usepackage{geometry}                		
\usepackage{graphicx}				
\usepackage{lmodern}    
\usepackage{ifthen}   
\usepackage[utf8]{inputenc}
\usepackage[T1]{fontenc}
\usepackage{amssymb, amsmath, amsthm, mathtools}
\usepackage{authblk}
\usepackage{mathrsfs} 

\newtheorem{theorem}{Theorem}[section]
\newtheorem{lemma}[theorem]{Lemma}

\newtheorem{proposition}[theorem]{Proposition}
\newtheorem{definition}[theorem]{Definition}

\numberwithin{equation}{section}

\title{The polynomial cluster value problem for Banach spaces}
\author{Isidro Humberto Munive Lima$^1$ \and Sofia Ortega Castillo$^2$}
\date{
Departamento de Matem\'aticas,\\
Centro Universitario de Ciencias Exactas e Ingenier\'ias,\\
Universidad de Guadalajara,\\
Blvd. Marcelino Garc\'ia Barrag\'an \#1421, Esq. Calzada Ol\'impica,\\
Guadalajara, 44430, Jalisco, M\'exico\\
\bigskip
$^1$ isidro.munive@academicos.udg.mx\\
\medskip
$^2$ sofia.ortega@academicos.udg.mx\\
\bigskip
\today
}							

\begin{document}
\maketitle

\begin{abstract}
We reduce the polynomial cluster value problem for the algebra of bounded analytic functions, $H^{\infty}$, on the ball of Banach spaces $X$ to the same polynomial cluster value problem for $H^{\infty}$ but on the ball of those spaces which are $\ell_1$-sums of finite dimensional spaces.
\end{abstract}

\section{Introduction}

The cluster value problem deals with the limit behavior at each boundary point of the domain of a bounded holomorphic function on the ball of a Banach space, where such limit is taken with respect to the weak-star topology as we embed such ball in the ball of its second dual. The cluster value problem has been studied in connection to the famous Corona problem, where the latter studies conditions that guarantee the density of a chosen domain in the spectrum of a Banach algebra of functions on such domain.

\medskip

In this work we focus on the Banach algebras of analytic functions on the open unit ball $B_X$ of a Banach space $X$ denoted by $H^{\infty}(B_X)$, $A_{\infty}(B_X)$ and $A_u(B_X)$, that we describe below. We also consider $H_b(X)$, which is a well known Fr\'echet algebra that we also describe below. Here $H^{\infty}(B_X)$ denotes the Banach algebra of bounded analytic functions on $B_X$, $A_{\infty}(B_X)$ denotes the Banach algebra of analytic functions on $B_X$ that extend continuously to the boundary of $B_X$, $A_u(B_X)$ is the Banach algebra of uniformly continuous analytic functions on $B_X$ and $H_b(X)$ consists of the entire functions on $X$ that are bounded on bounded sets of $X$.

\medskip

We will also use the notion of an $\ell_1$-sum of spaces: Given a family of Banach spaces $(X_{\alpha})_{\alpha\in I}$, its $\ell_1$-sum $(\sum X_{\alpha})_1$ is given by
$$(\sum X_{\alpha})_1=\{(x_{\alpha})_{\alpha\in I}\in\prod_{\alpha\in I} X_{\alpha}: \sum_{\alpha\in I}\|x_{\alpha}\|<\infty\},$$
and  it is known to be a Banach space since each $X_{\alpha}$ is so.

\medskip

The polynomial cluster value problem only differs of the original cluster value problem in that we replace the weak-star topology with the polynomial star topology. The polynomial cluster value problem for the Banach algebra $H^{\infty}(B_X)$ was first introduced in \cite{JO-CVPSCF} under the name "the cluster value problem for $H^{\infty}(B)$ over $A_u(B)$" and such problem was studied for spaces of continuous functions, as it was easier to work with this problem than the original one. Later on, the polynomial cluster value problem was studied in \cite{OP-PCVP} and its relevance was argued in depth; moreover, solutions were provided in some cases, including at boundary points of the ball of $\ell_1$.

\medskip

In this article, we show a reduction of the polynomial cluster value problem of $H^{\infty}(B_X)$ for Banach spaces $X$ to the polynomial cluster value problem on the ball of those spaces which are $\ell_1$-sums of finite dimensional spaces, as it was done in \cite{JO-CVPBS} but for the original cluster value problem. Hence, it remains essential to solve the polynomial cluster value problem in spaces such as $\ell_1$ and the finite dimensional $\ell_1^n$ with $n\in\mathbb{N}$, whose balls have been shown to be strongly pseudoconvex \cite{O-SPBS}, leading us to conjecture that a polynomial cluster value theorem holds true for those spaces since the polynomial cluster value problem for strongly pseudoconvex domains in $\mathbb{C}^n$ and sufficiently smooth boundary has a positive answer \cite{Mc}.
\medskip

Let us recall that the polynomial cluster value problem is posed as follows. Given an algebra $H$ of bounded analytic functions on the ball $B_X$ of a Banach space $X$ such that $H$ contains $A_u$ which is the algebra of uniformly continuous analytic functions, and given a point $x^{**}\in\bar{B}_X^{**}$, where $\bar{B}_X^{**}$ is the closed unit ball of the second dual $X^{**}$, we define, for $f\in H$, the polynomial cluster set $\text{Cl}_B^{\mathcal{P}}(f, x^{**})$ to be all the limit values of $f(x_{\alpha})$ over all nets $(x_{\alpha})$ in $B_X$ converging to $x^{**}$ in the polynomial-star topology, which is the topology generated by the Aron-Berner extensions of all polynomials on $X$ to $X^{**}$. In a finite dimensional space $X$ the (polynomial) cluster set simply recovers the boundary behavior of a bounded analytic function defined on its ball, since at each interior point it is a singleton. We now define the polynomial fiber $\text{M}_{x^{**}}^{\mathcal{P}}(H(B_X))$ as the non-zero multiplicative linear functionals on $H$ that, when restricted to $A_u$, coincide with the evaluation at $x^{**}$. The polynomial cluster value problem asks if $\text{Cl}_B^{\mathcal{P}}(f, x^{**})\supset \text{M}_{x^{**}}^{\mathcal{P}}(H(B_X)) (f)$. The other inclusion always holds as shown in \cite{OP-PCVP}. Moreover, the polynomial cluster value problem has been argued to be just as valid as the original cluster value problem \cite{OP-PCVP}, where it would not be surprising if the original cluster value problem always had a positive answer. 

\medskip

We reduce the polynomial cluster value problem for Banach spaces to $\ell_1$-sums of finite dimensional spaces for the algebra $H^{\infty}$ of all bounded analytic functions. We start with the separable case, proving a couple of lemmas for that end, see Lemma \ref{lemma2.1} and Lemma \ref{lemma2.2} herein. Lemma \ref{lemma2.1} is proven and posed slightly differently in \cite[Lemma~2.1]{JO-CVPBS} while Lemma \ref{lemma2.2} is an adaptation of Lemma~2.2 in \cite{JO-CVPBS} using instead that the Aron-Berner extension of a bounded holomorphic function on the ball of a Banach space which is an $l_1$-sum of finite dimensional spaces is again bounded and holomorphic on the ball of its second dual. We also work the nonseparable case. Moreover, we see that the method for $H=H^{\infty}$ does not apply to $H=A_{\infty}$ since the Aron-Berner extension of a function holomorphic on the ball of a Banach space that is also continuous on the closed ball results again holomorphic but is not necessarily continuous on the closed ball of the second dual. Indeed, by James' theorem there is a functional on a nonseparable Banach space $X$ that does not reach its norm on the closed ball and we compose it with an appropriate Blaschke product to obtain an holomorphic function on $A_{\infty}(B_X)$ such that its Aron-Berner extension is not in $A_{\infty}(B_{X^{**}})$. Finally, we pose the polynomial cluster value problem for $H_b(X)$ and we find it has a trivial positive answer.

\medskip

The reader may consult the following references on infinite dimensional holomorphy to have a better understanding of the topic. The definition of a polynomial on a Banach space can be found in \cite{M-CABS}. The reader can take a look at \cite{DG} to learn about the polynomial-star topology. The Aron-Berner extension was defined in \cite{AB} and several properties were established in \cite{G-AFBS}. Basic results related to the polynomial cluster value problem appear in \cite{OP-PCVP}.

\section{Reduction of the polynomial cluster value problem for $H^{\infty}$ to $\ell_1$-sums of finite dimensional spaces: separable case}

In order to show the desired reduction of the polynomial cluster value problem, we will need the next two lemmas. The first is a reformulation of what is shown in Lemma 2.1 in \cite{JO-CVPBS} while the second lemma is an adaptation of Lemma 2.2 in \cite{JO-CVPBS} to suit our needs in this work. Here, for $Z$ a Banach space and $f\in H^{\infty}(B_Z)$, $\tilde{f}$ denotes its Aron-Berner extension and if $S:Z_1\to Z_2$ is a bounded operator, $S^*$ denotes its adjoint.

\begin{lemma}\label{lemma2.1}
Let $Y$ be a separable Banach space and $Y_1\subset Y_2 \subset Y_3 \subset \dots$ an increasing sequence of finite dimensional subspaces whose union is dense in $Y$. Set $X=(\sum Y_n)_1$. Then the isometric quotient map $Q:X\to Y$ defined by $$Q(z_n)_n=\sum_{n=1}^{\infty} z_n$$ induces the isometric algebra homomorphism $Q^{\#}: H^{\infty}(B_Y)\to H^{\infty}(B_X)$ given by $$Q^{\#}(f)=f\circ Q, \quad \text{ for all }f\in H^{\infty}(B_Y),$$ and it satisfies that $Q^{\#}(A_u(B_Y))\subset A_u(B_X)$.
\end{lemma}

\begin{definition}\label{def2.2}
For $Y$ and $X$ as in Lemma \ref{lemma2.1}, let us define $S:X^*\to Y^*$ densely as in \cite{JO-CVPBS}: for $y\in \bigcup_k Y_k$ let $S_n(y)=(z_i)_i\in X$ be given by
$$z_i=\begin{cases} y, & \text{ if }i=n \text{ and } y\in Y_n\\ 0, & \text{ otherwise}\end{cases}$$
and if $\mathcal{U}$ is a free ultrafilter over $\mathbb{N}$ we consider for $x^*\in X^*$,
$$S(x^*)(y)=\lim_{n\in \mathcal{U}} x^*(S_n(y)).$$ In this way we obtain that $Sx^*$ is a bounded linear functional that extends to $Y$, and in turn $S$ is a bounded operator with norm bounded by $1$.
\end{definition}

\begin{lemma}\label{lemma2.2}
Under the assumptions of Lemma \ref{lemma2.1} and for $S$ the bounded operator constructed above, the function $T: H^{\infty}(B_X)\to H^{\infty}(B_Y)$ given by $$T(f)=\tilde{f}\circ S^{*}|_Y$$ is a norm one algebra homomorphism such that $T(X^*)\subset Y^*$, $T(A_u(B_X))\subset A_u(B_Y)$ and $T\circ Q^{\#}=I_{H^{\infty}(B_Y)}$.
\end{lemma}
\begin{proof}
Clearly $T$ is a well defined homomorphism with norm bounded by 1 since the Aron-Berner extension is a norm one homomorphism and $\|S^*|_Y\|\leq\|S^*\|=\|S\|\leq 1$. 

\medskip

Since for every $x^*\in X^*$, $T(x^*)=\widetilde{x^*}\circ S^*|_Y=j(x^*)\circ S^*|_Y$ is a bounded linear functional from $Y$ to $\mathbb{C}$ we have that $T(X^*)\subset Y^*$.

\medskip

Also, it is well known that the Aron-Berner extension of a uniformly continuous analytic function is still uniformly continuous and since $S^*|_Y$ is obviously uniformly continuous, we obtain $T(A_u(B_X))\subset A_u(B_Y)$.

\medskip

Now, given $f\in H^{\infty}(B_Y)$, $$T\circ Q^{\#}(f)=\widetilde{f\circ Q}\circ S^*|_Y.$$ 

\medskip

\textbf{Claim 1}: For all $f\in H^{\infty}(B_Y)$, $\widetilde{f\circ Q}=\tilde{f}\circ Q^{**}$.

\medskip

This follows from the fact that whenever $A\in L(^mY)$ and some nets $(z_{\alpha_i}^i)_{\alpha_i\in I_i}\subset X$ satisfy that $z_{\alpha_i}^i\xrightarrow{w^*} z_{i}^{**}\in X^{**}$ ($i=1, \dots, m$),
$$\widetilde{A\circ Q}(z_1^{**}, \dots, z_m^{**})=\lim_{\alpha_1\in I_1} \dots \lim_{\alpha_m\in I_m} A(Q z_{\alpha_1}^1, \dots, Qz_{\alpha_m}^m),$$
where $Q z_{\alpha_i}^i\xrightarrow{w^*} Q^{**}(z_i^{**})$ so 
$\widetilde{A\circ Q}(z_1^{**}, \dots, z_m^{**})=\tilde{A}(Q^{**}z_1^{**}, \dots, Q^{**}z_m^{**})$, and Claim 1 follows from using the Taylor series expansion of $\widetilde{f\circ Q}$.

\medskip

Therefore, $T\circ Q^{\#}(f)=\tilde{f}\circ Q^{**}\circ S^{*}|_Y$, for $f\in H^{\infty}(B_Y)$. 

\medskip

\textbf{Claim 2}: $Q^{**}\circ S^*=I_{Y^{**}}$.

\medskip

This holds since for $y^{**}\in Y^{**}$ and $y^*\in Y^*$,
\begin{align*}
Q^{**}(S^* y^{**})(y^*)&=[Q^{**}(y^{**}\circ S)](y^*)\\
&=y^{**}\circ S \circ Q^*(y^*)\\
&=y^{**}\circ S (y^*\circ Q),
\end{align*}
where for $y\in\bigcup_k Y_k$, 
$$S(y^*\circ Q)(y)=\lim_{n\in \mathcal{U}}y^*\circ Q (S_n y)=y^*(y),$$
so indeed $S(y^*\circ Q)=y^*$ and hence $Q^{**}(S^*y^{**})=y^{**}$, as we wanted.

\medskip

As a consequence it becomes clear that $T\circ Q^{\#}=I_{H^{\infty}(B_Y)}$ and thus 
$$\|T\|=\|T\| \|Q^{\#}\|\geq \|T\circ Q^{\#}\|=1$$
so $T$ is a norm one algebra homomorphism.
\end{proof}

\begin{theorem}
Let $Y$ be a separable Banach space and $Y_1\subset Y_2 \subset Y_3 \subset \dots$ an increasing sequence of finite dimensional subspaces whose union is dense in $Y$. Set $X=(\sum Y_n)_1$. If $H^{\infty}(B_X)$ satisfies the polynomial cluster value theorem at every $x^{**}\in \overline{B_{X^{**}}}$ then $H^{\infty}(B_Y)$ satisfies the polynomial cluster value theorem at every $y^{**}\in \overline{B_{Y^{**}}}$.
\end{theorem}
\begin{proof}

We know that $(M_{x^{**}}^{\mathcal{P}}(H^{\infty}(B_X)))(f)\subset \text{Cl}_{B_X}^{\mathcal{P}}(f, x^{**})$ for all $f\in H^{\infty}(B_X)$ and $x^{**}\in\overline{B_{X^{**}}}$. Let us show that $(M_{y^{**}}^{\mathcal{P}}(H^{\infty}(B_Y)))(g)\subset \text{Cl}_{B_Y}^{\mathcal{P}}(g, y^{**})$ for all $g\in H^{\infty}(B_X)$ and $y^{**}\in\overline{B_{Y^{**}}}$.

Let $y^{**}\in\overline{B_{Y^{**}}}$, $\tau\in M_{y^{**}}^{\mathcal{P}}(H^{\infty}(B_Y))$ and $g\in H^{\infty}(B_Y)$. Let $T$ be the algebra homomorphism from $H^{\infty}(B_X)$ to $H^{\infty}(B_Y)$ defined in Lemma \ref{lemma2.2}. From now on, $x^{**}:=S^*(y^{**})\in\overline{B_{X^{**}}}$, $f:=Q^{\#}(g)\in H^{\infty}(B_X)$ and $\varphi:=\tau\circ T$, which is clearly a non-zero multiplicative linear functional on $H^{\infty}(B_X)$.

\medskip

\textbf{Claim 3}: For all $h\in A_u(B_X)$, $\widetilde{Th}=\tilde{h}\circ S^*$.

\medskip

This follows from the density of $B_Y$ in $\overline{B_{Y^{**}}}$ in the polynomial-star topology. Indeed, if $(y_{\alpha})\subset B_Y$ converges to $y^{**}$ in the polynomial-star topology, then as we did in Claim 1 we can assume that $h$ is an $m$-homogeneous polynomial so $\tilde{h}\circ S^*$ is also an $m$-homogeneous polynomial, hence,
\begin{align*}
\widetilde{Th}(y^{**})&=\lim_{\alpha} Th(y_{\alpha})\\
&=\lim_{\alpha} \tilde{h} \circ S^*(y_{\alpha})\\
&=\tilde{h}\circ S^* (y^{**}).
\end{align*}

\medskip

Then we see that $\varphi\in M_{x^{**}}^{\mathcal{P}}(H^{\infty}(B_X))$ since for all $h\in A_u(B_X)$,
\begin{align*}
\varphi(h)&=\tau(Th)\\
&=\widetilde{Th}(y^{**})\\
&=\tilde{h}\circ S^*(y^{**})\\
&=\tilde{h}(x^{**})
\end{align*}

\medskip

Moreover, due to Lemma \ref{lemma2.2} we have that $T\circ Q^{\#}=I_{H^{\infty}(B_Y)}$, so
\begin{equation}\label{eq1}
\tau(g)=\tau(T\circ Q^{\#}g)=\varphi(Q^{\#}g)=\varphi(f).
\end{equation}

\textbf{Claim 4}: $\text{Cl}_{B_X}^{\mathcal{P}}(g\circ Q, x^{**})\subset \text{Cl}_{B_Y}^{\mathcal{P}}(g, Q^{**}x^{**})$.

\medskip

This follows from having that whenever $(x_{\alpha})$ converges in the polynomial-star topology to $x^{**}$, $Q(x_{\alpha})$ converges in the polynomial-star topology to $Q^{**}x^{**}$ since, as seen in Claim 1, $\widetilde{P\circ Q}=\tilde{P}\circ Q^{**}$ whenever $P$ is an $m$-homogeneous polynomial on $Y$.

\medskip

\textbf{Claim 5}: $Q^{**}x^{**}=y^{**}$.

\medskip

This equality follows from having $Q^{**}\circ S^*=I_{Y^{**}}$ as shown in Claim 2, indeed
\begin{align*}
Q^{**}(x^{**})=Q^{**}\circ S^{*}(y^{**})=y^{**}.
\end{align*}

\medskip 

It becomes now clear that 
\begin{equation}\label{eq2}
\tau(g)=\varphi(f)\in\text{Cl}_{B_X}^{\mathcal{P}}(f, x^{**})=\text{Cl}_{B_X}^{\mathcal{P}}(g\circ Q, x^{**})\subset \text{Cl}_{B_Y}^{\mathcal{P}}(g, Q^{**}x^{**})=\text{Cl}_{B_Y}^{\mathcal{P}}(g, y^{**}).
\end{equation}

\medskip

\end{proof}

\section{Reduction of the polynomial cluster value problem for $H^{\infty}$ to $\ell_1$-sums of finite dimensional spaces: nonseparable case}

As mentioned in \cite{JO-CVPBS}, the reduction of the polynomial cluster value problem for $H^{\infty}$ on the ball of a nonseparable Banach space to the same problem but for spaces that are $\ell_1$-sums of finite dimensional spaces is a straightforward generalization of the separable case. Yet, for the sake of completeness and for the convenience of the reader, we present here the construction and properties of $S$ corresponding to Definition \ref{def2.2} for the nonseparable case. After that, all that is left to show is Claim 8 in which we show a property of $S$ corresponding to Claim 2 for the nonseparable case. All the other constructions and proofs necessary to obtain the desired reduction in the nonseparable case  are proved in the exact same fashion as in the separable case.

\medskip

Let $Y$ be a Banach space with Hamel basis $B=\{e_{\lambda}:\lambda\in \Lambda\}$. Then,
\[
Y=\bigcup_{\substack{|A|<\infty\\ A\subset B}}Y_A,
\]
where $Y_A=\mathrm{span}\,A$ for each $A\subset B$. Let $I=\{A\subset B:|A|<\infty\}$ and set 
\[
X=\left(\sum_{A\in I}Y_{A}\right)_1.
\]
For each $A\in I$, we define
\[
C_A=\{A'\in I:A\subseteq A'\}.
\]
It is not difficult to show that the set
\[
\mathscr{A}=\{C\subseteq I: \text{ there exists }\,A\in I \,\text{such that $C\supseteq C_A$}\}
\]
is a filter. We let $\mathscr{A}_0$ be an ultrafilter over $I$ that contains $\mathscr{A}$.

 Now,  let $S:X^{\ast}\rightarrow Y^{\ast}$ be defined as follows. For $y\in Y$ and $x^{\ast}\in X^{\ast}$, we set 
\[
S_{A'}(y)=(z_A)_{A\in I}\in X,
\]
where
\[
z_A=
\begin{cases}
y, \,\text{if $y\in Y_{A'}$ and $A=A'$},\\
0,\, \text{otherwise}.
\end{cases}
\]

We define 
\begin{eqnarray*}
f_{x^{\ast},y}&:&I\rightarrow \overline{\mathbb{D}}(0,\|x^{\ast}\|\|y\|)\subset \mathbb{C}\\
&& A\rightarrow x^{\ast}(S_A(y)),
\end{eqnarray*}
since $|x^{\ast}(S_A(y))|\leq \|x^{\ast}\|\|y\|$. We then have that 
\begin{equation}\label{eq3.1}
S(x^{\ast})(y):=\lim_{\mathscr{A}_0}f_{x^{\ast},y}
\end{equation}
exists and it is unique by the ultrafilter convergence theorem. Now, let us consider a well-order on $I$, and we denote by $\min C$, for $C\subset I$, the minimum with respect to this well-order. Let $\mathscr{A}_0$ be partially ordered by reverse inclusion. We define a net by
\begin{eqnarray*}
z^{x^{\ast},y}&:&\mathscr{A}_0\rightarrow\mathbb{C}\\
&&C\rightarrow x^{\ast}(S_{\min C}(y)).
\end{eqnarray*}

\medskip

\textbf{Claim 6}. We have that $z^{x^{\ast},y}(C)\xrightarrow{C\in \mathscr{A}_0} S(x^{\ast})(y)$.

\medskip

Let $U_{x^{\ast},y}\subset\mathbb{C}$ be a neighborhood of $S(x^{\ast})(y)$. Then, by \eqref{eq3.1}
\[
C_0:=f^{-1}_{x^{\ast},y}(U_{x^{\ast},y})\in \mathscr{A}_0,\quad \text{i.e. $x^{\ast}(S_A(y))\in U_{x^{\ast},y}$ for every $A\in C_0$}.
\]
Therefore, if $C\subset C_0$, then $\min C\in C\subset C_0$, and hence we get
\[
x^{\ast}(S_{\min C}(y))\in U_{x^{\ast},y}.
\]

\medskip

\textbf{Claim 7}. We have that 
\[
\lim_{C\in\mathscr{A}_0}z^{x^{\ast},\mu y_1+y_2}(C)=\mu \lim_{C\in\mathscr{A}_0}z^{x^{\ast}, y_1}(C)+\lim_{C\in\mathscr{A}_0}z^{x^{\ast},y_2}(C),
\]
for every $\mu\in\mathbb{C}, y_1,y_2\in Y$ and $x^{\ast}\in X^{\ast}$.
\medskip

Indeed, if we set
\[
A_{y_i}:= \bigcap \{A\in I: y_i\in Y_A\}, \quad i=1,2,
\]
we get that if $A\in C_{A_{y_1}\cup A_{y_2}}:=C_{A_{y_1}}\cap C_{A_{y_2}}\in\mathscr{A}_0$,  then we also have $A\in C_{A_{\mu y_1+y_2}}$. Therefore,
\begin{equation}\label{eq3.2}
S_A(\mu y_1+y_2)=\mu S_A(y_1)+S_A(y_2).
\end{equation}

Let $\varepsilon>0$. By Claim 6, there exists $C_{\varepsilon}\in\mathscr{A}_0$ such that if $C\subset C_{\varepsilon}$,
\begin{eqnarray*}
|x^{\ast}(S_{\min C}(y_1))-S(x^{\ast})(y_1)|<\frac{\varepsilon}{3|\mu|+1},\\
|x^{\ast}(S_{\min C}(y_2))-S(x^{\ast})(y_2)|<\frac{\varepsilon}{3},\\
|x^{\ast}(S_{\min C}(\mu y_1+y_2))-S(x^{\ast})(\mu y_1+y_2)|<\frac{\varepsilon}{3}.
\end{eqnarray*}

Then, if $C\in\mathscr{A}_0$ satisfies that $C\subset C_{\varepsilon}\cap C_{A_{y_1}\cup A_{y_2}}$, we get that $\min C\in C\subset C_{A_{y_1}\cup A_{y_2}}$ so by \eqref{eq3.2}
$$
x^{\ast}(S_{\min C}(\mu y_1+y_2))=\mu x^{\ast}(S_{\min C}(y_1))+x^{\ast}(S_{\min C}(y_2)).
$$
and since also $C\subset C_{\varepsilon}$, we have that
\begin{eqnarray*}
&&|S(x^{\ast})(\mu y_1+y_2)-[\mu S(x^{\ast})(y_1)+S(x^{\ast})(y_2)]|\\
&\leq & |S(x^{\ast})(\mu y_1+y_2)-x^{\ast}(S_{\min C}(\mu_1 y_1+y_2))|\\
&&+|\mu x^{\ast}(S_{\min C}(y_1))-\mu S(x^{\ast})(y_1)|\\
&&+|x^{\ast}(S_{\min C}(y_2))-S(x^{\ast})(y_2)|\\
&<& \frac{\varepsilon}{3}+|\mu|\left(\frac{\varepsilon}{3|\mu|+1}\right)+\frac{\varepsilon}{3}<\varepsilon.
\end{eqnarray*}
Given that $\varepsilon>0$ was arbitrary, we get that $S(x^{\ast})(\mu y_1+y_2)=\mu S(x^{\ast})(y_1)+S(x^{\ast})(y_2)$. 

Therefore, $S(x^{\ast})$ is linear, and since
\begin{equation}
\label{normedSx}
\|S(x^{\ast})(y)\|\leq  \|x^{\ast}\|\|y\|, \quad \text{for every $y\in Y$},
\end{equation} 
we obtain $S(x^{\ast})\in X^{\ast}$.

Moreover, we also have that $S$ is linear. Indeed, let $\lambda\in \mathbb{C}, x^{\ast}_1,x^{\ast}_2\in X^{\ast}$ and $y\in Y$, then
\begin{eqnarray*}
S(\lambda x^{\ast}_1+x^{\ast}_2)(y)&=&\lim_{C\in\mathscr{A}_0}(\lambda x^{\ast}_1+x^{\ast}_2)(S_{\min C}(y))\\
&=&\lim_{C\in\mathscr{A}_0}\lambda x^{\ast}_1(S_{\min C}(y))+x^{\ast}_2(S_{\min C}(y))\\
&=&\lambda S(x^{\ast}_1)(y)+S(x^{\ast}_2)(y).
\end{eqnarray*}

By \eqref{normedSx}, we also get that $\|S(x^{\ast})\|\leq \|x^{\ast}\|$, i.e, $S$ is an operator with norm bounded by $1$.

\medskip

\textbf{Claim 8}. $Q^{\ast\ast}\circ S^{\ast}=I_{Y^{\ast\ast}}$.

\medskip

Indeed, if $y^{\ast}\in Y^{\ast}$ and $y\in Y$,
\begin{align*}
S(y^{\ast}\circ Q)(y)&=\lim_{\mathscr{A}_0} f_{y^{\ast}\circ Q, y}\\
&=\lim_{C\in \mathscr{A}_0} y^*\circ Q(S_{\min C}(y)).
\end{align*}
Now, if $C\subset C_{A_y}$ then $\min C\in C_{A_y}$ i.e. $y\in Y_{\min C}$, then $S_{\min C}(y)=(z_A)_{A\in I}$ where $$z_A=\begin{cases} y, &\text{if }A=\min C\\ 0, &\text{otherwise}\end{cases}$$
so $Q(S_{\min C}(y))=y$, obtaining that $y^*\circ Q(S_{\min C}(y))=y^*(y)$. Therefore, $S(y^{\ast}\circ Q)(y)=y^{\ast}(y)$ for every $y\in Y$, in other words
\begin{equation}
\label{Eq4.}
S(y^{\ast}\circ Q)=y^{\ast}.
\end{equation}

In consequence, for $y^{\ast\ast}\in Y^{\ast\ast}$ and $y^{\ast}\in Y^{\ast}$,
\begin{eqnarray*}
Q^{\ast\ast}(S^{\ast}y^{\ast\ast})(y^{\ast})&=& [Q^{\ast\ast}(y^{\ast\ast}\circ S)](y^{\ast})\\
&=& y^{\ast\ast}\circ S\circ Q^{\ast}(y^{\ast})\\
&=& y^{\ast\ast}\circ S(y^{\ast}\circ Q)\\
&=& y^{\ast\ast}(y^{\ast}) \qquad (\text{by \eqref{Eq4.}}).
\end{eqnarray*}

Hence, $Q^{\ast\ast}(S^{\ast}y^{\ast\ast})=y^{\ast\ast}$ for every $y^{\ast\ast}\in Y^{\ast\ast}$, in other words, $Q^{\ast\ast}\circ S^{\ast}=I_{Y^{\ast\ast}}$.

\section{The polynomial cluster value problem for $A_{\infty}(B_X)$ and $H_b(X)$}

An attempt to reduce the polynomial cluster value problem for $A_{\infty}(B_X)$ as we did in the previous section turns out unfruitful because the Aron-Berner extension of a function in $A_{\infty}(B_X)$, for $X$ a non-reflexive space such as an $\ell_1$-sum of finite dimensional spaces, is not necessarily again a function that extends continuously to the boundary of $B_X$. Let us prove this assertion via a construction based on \cite[Theorem~12.2]{ACG} and \cite[\S 2]{CGKM}.

\begin{proposition}
If $X$ is not reflexive then there exists $f\in A_{\infty}(B_X)$ such that its Aron Berner extension $\tilde{f}$ is not an element of $A_{\infty}(B_{X^{**}})$.
\end{proposition}
\begin{proof}
Let $X$ be a non-reflexive Banach space. Due to James' theorem, there exists $L\in X^*$ such that $\|L\|=1$ but $|L(x)|<1$ for all $x\in \overline{B}_X$.

\medskip

Meanwhile, $L^{**}\in X^{***}$ is $w^*$-continuous and since $\overline{B}_{X^{**}}$ is $w^*$-compact by the Banach-Alaoglu theorem we get that $L^{**}(\overline{B}_{X^{**}})$ is compact, in particular closed, so there exists $x_0^{**}\in \overline{B}_{X^{**}}$ such that $L^{**}(x_0^{**})=\|L^{**}\|=\|L\|=1$.

\medskip

Note that for $B$ a Blaschke product of the type
\begin{equation}\label{ABE:eq1}
B(\zeta)=\prod_{n=1}^{\infty} \frac{|\zeta_n|}{\zeta_n}\Big(\frac{\zeta_n-\zeta}{1-\overline{\zeta}_n\zeta}\Big) \text{ with }\sum_{n=1}^{\infty}(1-|\zeta_n|)<\infty \text{ and }1>|\zeta_n|> 0
\end{equation}
it holds that $|B(\zeta)|<1$ when $|\zeta|<1$ and we define $B(e^{i\theta})=\lim_{r\nearrow 1} B(r e^{i\theta})$ when it exists (this is well-defined for almost every $\theta$ due to a theorem of F.~Riesz~\cite{R}).

\medskip

Let us take $E=\{L^{**}(x_0^{**})=1\}$. Using \cite[Theorem~1]{C} and proceeding as in the proof of \cite[Theorem~2]{C} in the sense that we use \cite[p.~10]{Ca}, we have that the sequence $(\zeta_n)_n$ given by $\Big((1-1/2^{2n})e^{i/2^n}\Big)_n$ satisfies that $\zeta_n\to L^{**}(x_0^{**})$ and $B(L^{**}(x_0^{**}))$ is well-defined and has modulus one for $B$ the Blaschke product induced by $(\zeta_n)$ as in \eqref{ABE:eq1}.

\medskip

Let $f=B\circ L \in H^{\infty}(B_X)$. Let us see that $f\in A_{\infty}(B_X)$ but its Aron-Berner extension $\tilde{f}$ is not in $A_{\infty}(B_{X^{**}})$.

\medskip

Indeed, $f\in A_{\infty}(B_X)$ since whenever $x\in \overline{B}_X$ we have that $|L(x)|<1-\delta_x$ for some $\delta_x>0$, so that $|L(y)|<1$ when $y\in X$ and $\|y-x\|<\delta_x$. Therefore $f$ extends to be holomorphic and bounded in a neighborhood of $\overline{B}_X$ so in particular $f\in A_{\infty}(B_X)$.

\medskip

Let us now prove that $\tilde{f}=B\circ L^{**}$:

\medskip

Let $B(\zeta)=\sum_{n=0}^{\infty}\alpha_n \zeta^n$ be the Taylor series expansion of $B$ around zero, that converges uniformly on each $r\overline{\mathbb{D}}$ for $0<r<1$. Then, since $\|L\|=1$, for each $r\in(0,1)$ and for all $\epsilon>0$ there exists $N\in\mathbb{N}$ such that, when $x\in r \overline{B}_X$,

$$|B(L(x))-\sum_{n=0}^m \alpha_n (L(x))^n|<\epsilon \text{ when }m\geq N,$$
i.~e. $$|f(x)-\sum_{n=0}^{m}\alpha_n L^n(x)|<\epsilon.$$

Then, since the Aron-Berner extension is an isometric algebra homomorphism, we have that, when $x^{**}\in r\overline{B}_{X^{**}}$,
$$|\tilde{f}(x^{**})-\sum_{n=0}^{m}\alpha_n (L^{**}(x^{**}))^n|\leq\epsilon<2\epsilon,$$
i.~e. $$\tilde{f}=\sum_{n=0}^{\infty}\alpha_n (L^{**})^n=B\circ L^{**}.$$

\medskip

Now, take nonnegative reals $r_n\nearrow 1$ and let $\lambda_n=r_n=r_n L^{**}(x_0^{**})$ for each $n\in \mathbb{N}$. Then $|\lambda_n-\zeta_n|\to 0$ while $|B(\lambda_n)-B(\zeta_n)|=|B(\lambda_n)|\to 1$. And $y_n=\lambda_n x_0^{**}$ and $z_n=\zeta_n x_0^{**}$ are elements of $B_{X^{**}}$ that converge to $x_0^{**}$ yet 
$$|\tilde{f}(y_n)-\tilde{f}(z_n)|=|B(\lambda_n)|\to 1,$$\
so $\tilde{f}\notin A_{\infty}(B_{X^{**}})$.
\end{proof}

Let us conclude this section introducing the polynomial cluster value problem for $H_b(X)$ and proving that we clearly have a cluster value theorem for $H_b(X)$ for every Banach space $X$. 

\medskip

Comparing with the definition of the cluster value problem at the end of \cite{CGMS} for the algebra of analytic functions that are bounded on bounded sets, let us define for $r>0$, $f\in H_b(X)$ and $z\in r\overline{B}_{X^{**}}$, the polynomial $r$-cluster set of $f$ at $z$ as
$$Cl_r(f,z)=\{\lambda\in\mathbb{C}:\text{ there exists }(x_{\alpha})\subset rB_X \text{ such that }x_{\alpha}\xrightarrow{\mathcal{P}} z, f(x_{\alpha})\to\lambda\}$$
where $\mathcal{P}$ denotes the polynomial-star topology on $X^{**}$, also denoted by $\sigma(X^{**}, \mathcal{P}(X))$.

\medskip

Also, given $z\in X^{**}$ and $r>0$, let us define the polynomial $r$-fiber of $\mathcal{M}(H_b(X))$ over $z$ as
$$\mathcal{M}_{z,r}=\{\varphi \in \mathcal{M}(H_b(X)): \pi(\varphi)=z, \varphi\prec rB_X\}$$
where $\pi: \mathcal{M}(H_b(X))\to \mathcal{P}(X)^*$ is given by $\pi(\varphi)=\varphi|_{\mathcal{P}(X)}$ while $\varphi\prec rB_X$ means that $|\varphi(f)|\leq\|f\|_{rB_X}$ for every $f\in H_b(X)$.

\medskip

We will say that there is a cluster value theorem for $H_b(X)$ if $Cl_r(f,z)=\hat{f}(\mathcal{M}_{z,r})$ for every $r>0$, $f\in H_b(X)$ and $z\in r\overline{B}_{X^{**}}$, where $\hat{f}:\mathcal{M}(H_b(U))\to\mathbb{C}$ denotes the Gelfand transform of $f$ given by $\hat{f}(\varphi)=\varphi(f)$.

\medskip

Note that, whenever $f\in H_b(X)$ and $r>0$, we have that $f\in H^{\infty}((r+\epsilon)B_X)$ for any $\epsilon>0$, so $f\in A_u(rB_X)$ due to Schwarz' lemma \cite[Thm.~7.19]{M-CABS} and the convexity of $(r+\epsilon)B_X$. Now, if $z\in r\overline{B}_{X^{**}}$ and $x_{\alpha}\xrightarrow{\mathcal{P}} z$, then $f(x_{\alpha})\to_{\alpha} \tilde{f}(z)$ since $f$ is the uniform limit of polynomials on $rB_X$. Consequently, $Cl_r(f,z)=\{\tilde{f}(z)\}$.

\medskip

Meanwhile, if $f\in H_b(X)$, due to \cite[Theorem~7.13]{M-CABS} we have that the radius of convergence of the Taylor series of $f$ at 0 is $\infty$, so for every $x\in X$,
\begin{equation}\label{ABE:eq2}
f(x)=\sum_{m=0}^{\infty} P^m f(0),
\end{equation}
with uniform convergence in $rB_X$ for every $r>0$, where $P^m f(0)$ is an $m$-homogeneous polynomial on $X$. Then, if also $r>0$ and $z\in r\overline{B}_{X^{**}}$ are fixed, for $\phi\in\mathcal{M}_{z,r}$ we have that $\hat{f}(\mathcal{M}_{z,r})=\{\tilde{f}(z)\}$ since the partial sums in \eqref{ABE:eq2} belong to $\mathcal{P}(X)$ and converge to $f$ uniformly on $rB_X$ so $\varphi(f)=\sum_{m=0}^{\infty} \varphi(P^m f(0))=\sum_{m=0}^{\infty}\widetilde{P^m f(0)}(z)=\tilde{f}(z)$.

\medskip

So, indeed $Cl_r(f,z)=\hat{f}(\mathcal{M}_{z,r})$ for every $r>0$, $f\in H_b(X)$ and $z\in r\overline{B}_{X^{**}}$. Actually we proved that such sets are trivial, equal to $\{\tilde{f}(z)\}$.

\end{document}